\documentclass[11pt,a4paper,twoside,groupcitations]{article}
\usepackage[T1]{fontenc}
\usepackage[ansinew]{inputenc}
\usepackage[english]{babel}
\usepackage{amsfonts}
\usepackage{amsmath}
\usepackage{bm}
\usepackage{array}
\usepackage{amsthm}
\usepackage{amssymb}
\usepackage{graphicx}
\usepackage{subfigure}
\usepackage{braket}
\usepackage{verbatim}
\usepackage[table]{xcolor}
\usepackage{caption}
\usepackage{cite}
\usepackage{textcomp}
\usepackage{url}
\usepackage{multicol}
\usepackage{tikz}
\usetikzlibrary{positioning,arrows}
\usetikzlibrary{decorations.pathmorphing}
\usetikzlibrary{decorations.markings}
\usetikzlibrary{calc,decorations.markings}
\usetikzlibrary{arrows,shapes}
\usetikzlibrary{matrix,arrows}
\usepackage{pgfplots}
\usepackage{xparse}
\newcommand{\be}{\begin{equation}}
	\newcommand{\ee}{\end{equation}}
\definecolor{pinegreen}{rgb}{0.0, 0.47, 0.44}

\theoremstyle{definition}

\newtheorem{te}{Theorem}

\newtheorem{prop}{Proposition}

\theoremstyle{remark}
\newtheorem*{os}{Remark}

\title{\bf A note on nonlinear diffusive equations on Poincar\'e half plane} %

\date{}
\author{Roberto~Garra, Francesco~Maltese
}
\begin{document}
	\maketitle

	\begin{abstract}
		In this paper we show some explicit results regarding non-linear diffusive equations on Poincar\'e half plane. We obtain exact solutions by using the generalized separation of variables and we
		also show the meaning of these results in the context of the general theory of the invariant subspace method.
		
		\bigskip
		
		\textit{Keywords:} Nonlinear diffusion equation, Hyperbolic geometry, Exact solutions. 
	\end{abstract}

	\section{Introduction}
	The analysis of exact solutions for nonlinear diffusive equations plays a central role for the applications, for example in the physics of porous media \cite{vaz}.
	There are many different methods to find particular exact interesting solutions for nonlinear PDEs, for example by using the invariant subspace method \cite{gala} or
	the generalized separation of variables \cite{polyanin,polyanin1}.
	In the recent paper \cite{noi}, the authors study the non-linear time-fractional equation
	\begin{equation}\label{1}
		\frac{\partial^\nu }{\partial t^\nu} u (\eta,t)= \Delta_H  u^{n}-u = \frac{1}{\sinh \eta}\frac{\partial}{\partial \eta}\sinh \eta \frac{\partial  u^{n}}{\partial \eta}-u, \quad n>1,
	\end{equation}
involving the time-fractional derivative $\partial^\nu/\partial t^\nu$ in the sense of Caputo. 
   Motivated by this study, here we consider two general class of nonlinear equations in the hyperbolic plane, showing that it is possible to construct interesting 
   exact solutions for these equations by using simple methods.\\
   We first consider a more general formulation of the equation \eqref{1}, that is 
   \begin{equation}\label{2in}
   	\widehat{O}_t  u (\eta,t)= \Delta_H  u^{n}-u = \frac{1}{\sinh \eta}\frac{\partial}{\partial \eta}\sinh \eta \frac{\partial  u^{n}}{\partial \eta}-u, \quad n>1,
   \end{equation}
   where $\widehat{O}_t$ is a linear differential operator acting on the variable $t$. We show that this equation admits elementary exact solutions that can be constructed by separation of
   variables. These results can be proved rigorously by using the invariant subspace method for example. We consider some particular interesting cases in the wide family of nonlinear
   equations \eqref{2in}, when the linear operato $\widehat{O}_t$ is a first order time-derivative or a Laguerre derivative or a fractional derivative. 
   In the second part of the paper we consider the following family of nonlinear equations in the hyperbolic plane
	\begin{equation}
		\widehat{O}_t u (\eta,t) = u\Delta_H u = \frac{u}{\sinh \eta}\frac{\partial}{\partial \eta}\sinh \eta \frac{\partial u}{\partial \eta}.
	\end{equation}
  This case is particularly interesting, since it admits useful non trivial exact solutions.\\
  Essentially there are few investigations about nonlinear fractional equations on the hyperbolic plane. Our main aim is to provide some ideas for the construction of 
  exact solutions for some family of interesting nonlinear equations. By using similar methods, many other problems can be solved.

\section{Construction of exact solutions for nonlinear diffusive equations on Poincar\'e half plane}
\subsection{A first interesting case}
There are many families of nonlinear equations that can be solved by means of the invariant subspace method (see \cite{gala} for the details). As far as we know, there are few studies about the application of this method to solve nonlinear equations in the hyperbolic space. We show the potential utility of this method, starting from a generalization of the time-fractional nonlinear diffusive equation on Poincair\'e half plane studied in \cite{noi}, that is

\begin{equation}\label{3.5}
	\widehat{O}_t u (\eta,t)= \Delta_H  u^{n}-u = \frac{1}{\sinh \eta}\frac{\partial}{\partial \eta}\sinh \eta \frac{\partial  u^{n}}{\partial \eta}-u, \quad n>1 .
\end{equation} 

We have the following result.

\begin{te}
The equation \eqref{3.5} admits a solution of the form	
	\begin{equation}\label{2}
		u(\eta,t) = f(t)\sqrt[n]{c_1\ln(\tanh \frac{\eta}{2})+c_2}, 
	\end{equation}
where $f(t)$ is a solution of the equation
$$ \widehat{O}_t f(t)=-f(t) $$
	and  $c_1$ e $c_2$ are real  constants such that
	$c_1\ln(\tanh \frac{\eta}{2})+c_2>0$.
\end{te}
\begin{proof}
	We observe that	
	
	$$	\frac{\partial}{\partial \eta}\sinh \eta \frac{\partial}{\partial \eta}\bigg[  f(t)\sqrt[n]{c_1\ln(\tanh \frac{\eta}{2})+c_2}\bigg]^{n}= \frac{\partial}{\partial \eta}\sinh \eta \frac{\partial}{\partial \eta}\bigg[ f^{n}(t)(c_1\ln(\tanh \frac{\eta}{2})+c_2)\bigg ]$$

	$$ =f^{n}(t) \frac{\partial}{\partial \eta}\sinh \eta \frac{\partial}{\partial \eta}\bigg[ c_1\ln(\tanh \frac{\eta}{2})+c_2\bigg ]=0$$

	Thus in this case $u(\eta,t)$ belongs to kernel of $\Delta_H$.
	From this observation, we can insert the function $u(\eta,t)$ in the equation \eqref{3.5} and we have 
	
	$$	\widehat{O}_t \bigg( f(t)\sqrt[n]{c_1\ln(\tanh \frac{\eta}{2})+c_2}\bigg)=\frac{1}{\sinh \eta}\frac{\partial}{\partial \eta}\sinh \eta \frac{\partial }{\partial \eta}\bigg[  f(t)\sqrt[n]{c_1\ln(\tanh \frac{\eta}{2})+c_2}\bigg]^{n}-f(t)\sqrt[n]{c_1\ln(\tanh \frac{\eta}{2})+c_2} $$
	
	$$\left(\sqrt[n]{c_1\ln(\tanh \frac{\eta}{2})+c_2}\right) \cdot \widehat{O}_t  f(t)=-f(t)\sqrt[n]{c_1\ln(\tanh \frac{\eta}{2})+c_2} $$
	
	$$\widehat{O}_t  f(t)=-f(t)$$
	By hypothesis on $f(t)$ we obtain the claimed result.
\end{proof}
By means of Theorem 2.1 we can provide an exact solution for \eqref{3.5} in different interesting cases.
For example if $\widehat{O}_t$ coincides with a Caputo fractional derivative of order
$\beta \in (0,1)$ , i.e. in this case we have that 
\begin{equation}
	\widehat{O}_t u(\eta,t) = \frac{\partial^\beta u}{\partial t^\beta}= \frac{1}{\Gamma(1-\beta)}\int_0^t (t-\tau)^{-\beta}\frac{\partial u}{\partial \tau}d\tau,
\end{equation}
 we have that the function
\begin{equation}\label{a}
	u(\eta,t) = E_{\beta}(-t^{\beta})\sqrt[n]{c_1\ln(\tanh \frac{\eta}{2})+c_2},
\end{equation}
is a solution for the non-linear time-fractional equation
\begin{equation}\label{aa}
	\frac{\partial^\beta}{\partial t^\beta} u (\eta,t)= \frac{1}{\sinh \eta}\frac{\partial}{\partial \eta}\sinh \eta \frac{\partial  u^{n}}{\partial \eta}-u, \quad n>1 .
\end{equation} 
In \eqref{a} we denoted by $E_\beta(-t^\beta)$ the one-parameter Mittag-Leffler function, i.e. the eigenfunction of the Caputo fractional derivative (see e.g. \cite{main}).\\
For $\beta = 1 $ we recover the solution 
\begin{equation}\label{abis}
	u(\eta,t) = e^{-t}\sqrt[n]{c_1\ln(\tanh \frac{\eta}{2})+c_2},
\end{equation}
for the classical equation 
\begin{equation}\label{aas}
	\frac{\partial}{\partial t} u (\eta,t)= \frac{1}{\sinh \eta}\frac{\partial}{\partial \eta}\sinh \eta \frac{\partial  u^{n}}{\partial \eta}-u, \quad n>1 .
\end{equation} 

Another interesting case is when $\widehat{O}_t = \frac{\partial}{\partial t}t \frac{\partial}{\partial t}$ that is the so-called Laguerre derivative (see for example \cite{ricci}).
In this case we have that the function

\begin{equation}\label{c0}
	u(\eta,t) = C_0(t)\sqrt[n]{c_1\ln(\tanh \frac{\eta}{2})+c_2},
\end{equation}
is a solution for the equation
\begin{equation}
\frac{\partial}{\partial t}t \frac{\partial}{\partial t} u (\eta,t)= \frac{1}{\sinh \eta}\frac{\partial}{\partial \eta}\sinh \eta \frac{\partial  u^{n}}{\partial \eta}-u, \quad n>1 .
\end{equation} 
In the Equation \eqref{c0} the function $C_0(t)$ is a Bessel-type function, i.e. 
\begin{equation}
	C_0(t) = \sum_{k=0}^\infty \frac{t^k}{k!^2}
\end{equation}

By using this approach we can construct exact solutions for many different nonlinear PDEs involving the hyperbolic Laplacian. For example, an interesting outcome inspired by \cite{calo} is given by the following result
\begin{prop}
	The nonlinear equation
	\begin{equation}\label{0}
		\frac{\partial u}{\partial t}-\frac{i\omega}{\alpha}u = \frac{1}{\sinh \eta}\frac{\partial}{\partial \eta}\sinh \eta \frac{\partial  u^{n}}{\partial \eta},
	\end{equation}
admits the completely periodic separating variable solution
\begin{equation}\label{00}
	u(x,t)= \exp\bigg(\frac{i\omega}{\alpha}t\bigg)\sqrt[n]{c_1\ln(\tanh \frac{\eta}{2})+c_2}.
\end{equation}
\end{prop}
The proof can be simply obtained by direct substitution.

\subsection{The second case}

    \begin{te}
    	The nonlinear diffusion equation 
    	\begin{equation}\label{top}
    		\widehat{O}_t u (\eta,t) = \frac{u}{\sinh \eta}\frac{\partial}{\partial \eta}\sinh \eta \frac{\partial u}{\partial \eta}
    	\end{equation}
    admits an explicit solution of the form 
    \begin{equation}
    	u(\eta,t) = f_1(t)\ln(\sinh \eta)+f_2(t)\ln(\tanh \frac{\eta}{2})+f_3(t),
    \end{equation}
where $f_1(t)$ is a solution of the nonlinear equation  
\begin{equation}
	\widehat{O}_t f_1(t) = f_1^2(t),
\end{equation}
while $f_2$ and $f_3$ satisfy the equations
\begin{equation}
		\widehat{O}_t f_2 = f_1 f_2, \quad \widehat{O}_t f_3 =f_1 f_3. 
\end{equation}
    	
    \end{te}

\begin{proof}
	
We observe that 
\begin{equation}
\frac{\partial}{\partial \eta}\sinh \eta \frac{\partial}{\partial \eta}\bigg[f_2(t)\ln(\tanh \frac{\eta}{2})+f_3(t)\bigg] = 0,
\end{equation}
and 
\begin{equation}
	\frac{f_1(t)\ln(\sinh \eta)}{\sinh \eta}\frac{\partial}{\partial \eta}\sinh \eta \frac{\partial f_1 \ln(\sinh \eta)}{\partial \eta} = f_1^2 \ln(\sinh \eta).
\end{equation}
Therefore, if we search a solution of the form
\begin{equation}\label{sol}
u(\eta,t) = f_1(t)\ln(\sinh \eta)+f_2(t)\ln(\tanh \frac{\eta}{2})+f_3(t)
\end{equation}
	we have, by substitution 
	\begin{equation}
\ln(\sinh \eta)\widehat{O}_t f_1(t)+ \ln(\tanh \frac{\eta}{2})\widehat{O}_t f_2(t)+\widehat{O}_t f_3(t) = f_1^2 \ln(\sinh \eta)+f_1f_2\ln(\tanh \frac{\eta}{2})+f_1f_3.
	\end{equation}
Therefore the function \eqref{sol} is a solution if 
$$\widehat{O}_t f_1(t) = f_1^2, \quad 	\widehat{O}_t f_2 = f_1 f_2, \quad \widehat{O}_t f_3 =f_1 f_3.$$
as stated.
	\end{proof}
\begin{os}
Following the theory of the invariant subspace method, we have used the fact that the equation \eqref{top} admits the invariant subspace $W^3:<1,\ln(\sinh \eta),\ln(\tanh \frac{\eta}{2})>$.
\end{os}

We now consider some interesting particular cases that are related to classical mathematical models.
First of all we consider the classical nonlinear diffusive case
\begin{equation}\label{aaf}
		\frac{\partial u}{\partial t} = \frac{u}{\sinh \eta}\frac{\partial}{\partial \eta}\sinh \eta \frac{\partial u}{\partial \eta}.
\end{equation}
In view of the previous Theorem, we can construt an exact solution by solving the following simple ODEs
$$\frac{d}{dt}f_1(t) = f_1^2, \quad \frac{d}{dt}f_2 =  f_1 f_2, \quad \frac{d}{dt}f_3 = f_1 f_3.$$
Therefore we have that a solution of the equation \eqref{aaf} is given by 
\begin{equation}
u(\eta,t) = \frac{\ln(\sinh \eta)}{t_0-t}+c_1\frac{\ln(\tanh \frac{\eta}{2})}{t_0-t}+\frac{c_2}{t_0-t}.
\end{equation}

\end{document}